\documentclass[12pt,fleqn]{article}

\usepackage{amsmath,amssymb,amsthm}
\usepackage{geometry} 
\usepackage[pdfpagemode=UseNone,pdfstartview=FitH]{hyperref}

\newcommand{\abs}[1]{\left|#1\right|}
\newcommand{\bdry}[1]{\partial #1}
\newcommand{\dint}{\ds{\int}}
\newcommand{\ds}[1]{\displaystyle #1}
\newcommand{\dualp}[3][]{\left(#2,#3\right)_{#1}}
\newcommand{\eps}{\varepsilon}
\newcommand{\norm}[2][]{\left\|#2\right\|_{#1}}
\renewcommand{\o}{\text{o}}
\newcommand{\PS}[1]{$(\text{PS})_{#1}$}
\newcommand{\pnorm}[2][]{\if #1'' \left|#2\right|_p \else \left|#2\right|_{#1} \fi}
\newcommand{\QED}{\mbox{\qedhere}}
\newcommand{\R}{\mathbb R}
\newcommand{\seq}[1]{\left(#1\right)}
\newcommand{\set}[1]{\left\{#1\right\}}

\DeclareMathOperator{\divg}{div}

\newenvironment{enumroman}{\begin{enumerate}

}{\end{enumerate}}

\newtheorem{corollary}{Corollary}[section]
\newtheorem{lemma}[corollary]{Lemma}
\newtheorem{proposition}[corollary]{Proposition}
\newtheorem{theorem}[corollary]{Theorem}

\theoremstyle{remark}
\newtheorem{remark}[corollary]{Remark}

\numberwithin{equation}{section}

\title{\bf On the Brezis-Nirenberg problem for the $(p,q)$-Laplacian}
\author{\bf Ky Ho\\
Department of Mathematical Sciences\\
Ulsan National Institute of Science and Technology\\
Ulsan 44919, Republic of Korea\\
\em kyho@unist.ac.kr\\
[\bigskipamount]
\bf Kanishka Perera\\
Department of Mathematical Sciences\\
Florida Institute of Technology\\
Melbourne, FL 32901, USA\\
\em kperera@fit.edu\\
[\bigskipamount]
\bf Inbo Sim\\
Department of Mathematics\\
University of Ulsan\\
Ulsan 44610, Republic of Korea\\
\em ibsim@ulsan.ac.kr}
\date{}

\begin{document}

\maketitle

\begin{abstract}
We prove some existence and nonexistence results for a class of critical $(p,q)$-Laplacian problems in a bounded domain. Our results extend and complement those in the literature for model cases.
\end{abstract}

\newpage

\section{Introduction}

Consider the critical $(p,q)$-Laplacian problem
\begin{equation} \label{001}
\left\{\begin{aligned}
- \Delta_p\, u - \Delta_q\, u & = b\, |u|^{s-2}\, u + |u|^{p^\ast - 2}\, u && \text{in } \Omega\\[10pt]
u & = 0 && \text{on } \bdry{\Omega},
\end{aligned}\right.
\end{equation}
where $\Omega$ is a bounded domain in $\R^N,\, N \ge 2$, $1 < q < p < N$, $p^\ast = Np/(N - p)$ is the critical Sobolev exponent, $1 < s < p^\ast$, and $b > 0$. It was shown in Li and Zhang \cite{MR2509998} that this problem has infinitely many solutions when $1 < s < q$ and $b > 0$ is sufficiently small. On the other hand, it was shown in Yin and Yang \cite{MR2890966} that it has a nontrivial solution when $p < s < p^\ast$ and $b > 0$ is sufficiently large. Sufficient conditions for the existence of a nontrivial solution when $s = p$ and $b > 0$ is either small or large were given in Candito et al.\! \cite{MR3415031}. A rescaling of a result in Ho and Sim \cite{MR4131823} shows that the related problem
\begin{equation} \label{002}
\left\{\begin{aligned}
- \Delta_p\, u - \nu\, \Delta_q\, u & = b\, |u|^{s-2}\, u + |u|^{p^\ast - 2}\, u && \text{in } \Omega\\[10pt]
u & = 0 && \text{on } \bdry{\Omega}
\end{aligned}\right.
\end{equation}
has a nontrivial solution when $q < s < p$ and $\nu, b > 0$ are sufficiently small. The borderline case $s = q$ does not seem to have been studied in the literature.

In the present paper we prove some existence results for a more general class of critical $(p,q)$-Laplacian problems that, in particular, give a nontrivial solution of problem \eqref{001} for all $b > 0$ and a nontrivial solution of problem \eqref{002} for sufficiently small $\nu > 0$ and all $b > 0$. More specifically, our main results for the model problems \eqref{001} and \eqref{002} are the following:

\begin{theorem}
Problem \eqref{001} has a nontrivial weak solution for all $b > 0$ in each of the following cases:
\begin{enumroman}
\item $1 < q < N(p - 1)/(N - 1)$ and $N^2 (p - 1)/(N - 1)(N - p) < s < p^\ast$,
\item $N(p - 1)/(N - 1) \le q < p$ and $Nq/(N - p) < s < p^\ast$.
\end{enumroman}
In particular, problem \eqref{001} has a nontrivial weak solution for all $b > 0$ when $N^2 - p\, (p + 1) N + p^2 \ge 0$, $q \le (N - p)\, p/N$, and $p < s < p^\ast$, and when $N^2 - p\, (p + 1) N + p^2 > 0$, $q < (N - p)\, p/N$, and $s = p$.
\end{theorem}

\begin{theorem}
There exists $\nu_0 > 0$ such that problem \eqref{002} has a nontrivial weak solution for all $\nu \in (0,\nu_0)$ and $b > 0$ in each of the following cases:
\begin{enumroman}
\item $N \ge p^2$ and $q < s < p^\ast$,
\item $N < p^2$ and either $q < s < p$ or $(Np - 2N + p)\, p/(N - p)(p - 1) < s < p^\ast$.
\end{enumroman}
In particular, problem \eqref{002} has a nontrivial weak solution for all $\nu \in (0,\nu_0)$ and $b > 0$ when $q < s < p$, and when $N \ge p^2$ and $s = p$.
\end{theorem}

In the borderline case $s = q$ we show that problem \eqref{001} has no nontrivial weak solution for all sufficiently small $b > 0$ when $\Omega$ is a star-shaped domain with $C^1$-boundary (see Theorem \ref{Theorem 106}). The proof of this nonexistence result will be based on a new Poho\v{z}aev type identity for the $(p,q)$-Laplacian (see Theorem \ref{Theorem 108}), which is of independent interest.

We refer the reader to Marano and Mosconi \cite{MR3732167} for a survey of recent existence and multiplicity results for subcritical and critical $(p,q)$-Laplacian problems in bounded domains.

\section{Statement of results}

We consider the critical $(p,q)$-Laplacian problem
\begin{equation} \label{101}
\left\{\begin{aligned}
- \Delta_p\, u - \Delta_q\, u & = f(x,u) + |u|^{p^\ast - 2}\, u && \text{in } \Omega\\[10pt]
u & = 0 && \text{on } \bdry{\Omega},
\end{aligned}\right.
\end{equation}
where $\Omega$ is a bounded domain in $\R^N,\, N \ge 2$, $1 < q < p < N$, $p^\ast = Np/(N - p)$ is the critical Sobolev exponent, and $f$ is a Carath\'eodory function on $\Omega \times \R$ satisfying
\begin{equation} \label{102}
f(x,0) = 0 \quad \text{for a.a.\! } x \in \Omega
\end{equation}
and the subcritical growth condition
\begin{equation} \label{103}
|f(x,t)| \le a_1\, |t|^{r-1} + a_2 \quad \text{for a.a.\! } x \in \Omega \text{ and all } t \in \R
\end{equation}
for some constants $a_1, a_2 > 0$ and $r \in (p,p^\ast)$. A weak solution of this problem is a function $u \in W^{1,\,p}_0(\Omega)$ satisfying
\[
\int_\Omega \left(|\nabla u|^{p-2}\, \nabla u \cdot \nabla v + |\nabla u|^{q-2}\, \nabla u \cdot \nabla v - f(x,u)\, v - |u|^{p^\ast - 2}\, uv\right) dx = 0 \quad \forall v \in W^{1,\,p}_0(\Omega),
\]
where $W^{1,\,p}_0(\Omega)$ is the usual Sobolev space with the norm $\norm{u} = \left(\int_\Omega |\nabla u|^p\, dx\right)^{1/p}$. Weak solutions coincide with critical points of the $C^1$-functional
\[
E(u) = \int_\Omega \left(\frac{1}{p}\, |\nabla u|^p + \frac{1}{q}\, |\nabla u|^q - F(x,u) - \frac{1}{p^\ast}\, |u|^{p^\ast}\right) dx, \quad u \in W^{1,\,p}_0(\Omega),
\]
where $F(x,t) = \int_0^t f(x,\tau)\, d\tau$ is the primitive of $f$. Recall that a sequence $\seq{u_j} \subset W^{1,\,p}_0(\Omega)$ such that $E(u_j) \to c$ and $E'(u_j) \to 0$ is called a \PS{c} sequence. Let
\begin{equation} \label{104}
c^\ast = \frac{1}{N}\, S^{N/p},
\end{equation}
where
\begin{equation} \label{105}
S = \inf_{u \in W^{1,\,p}_0(\Omega) \setminus \set{0}}\, \frac{\dint_\Omega |\nabla u|^p\, dx}{\left(\dint_\Omega |u|^{p^\ast} dx\right)^{p/p^\ast}}
\end{equation}
is the best Sobolev constant. If $0 < c < c^\ast$, then every \PS{c} sequence has a subsequence that converges weakly to a nontrivial critical point of $E$ (see Proposition \ref{Proposition 201}).

Let
\begin{equation} \label{106}
\lambda_1 = \inf_{u \in W^{1,\,p}_0(\Omega) \setminus \set{0}}\, \frac{\dint_\Omega |\nabla u|^p\, dx}{\dint_\Omega |u|^p\, dx}, \qquad \mu_1 = \inf_{u \in W^{1,\,q}_0(\Omega) \setminus \set{0}}\, \frac{\dint_\Omega |\nabla u|^q\, dx}{\dint_\Omega |u|^q\, dx}
\end{equation}
be the first Dirichlet eigenvalues of the $p$-Laplacian and the $q$-Laplacian, respectively. Assume that
\begin{equation} \label{107}
F(x,t) \le \frac{\lambda}{p}\, |t|^p + \frac{\mu_1}{q}\, |t|^q \quad \text{for a.a.\! } x \in \Omega \text{ and } |t| < \delta
\end{equation}
for some $\lambda \in (0,\lambda_1)$ and $\delta > 0$. It follows from this and \eqref{103} that
\[
F(x,t) \le \frac{\lambda}{p}\, |t|^p + \frac{\mu_1}{q}\, |t|^q + a_3\, |t|^r \quad \text{for a.a.\! } x \in \Omega \text{ and all } t \in \R
\]
for some constant $a_3 > 0$, so
\[
E(u) \ge \int_\Omega \left[\frac{1}{p} \left(1 - \frac{\lambda}{\lambda_1}\right) |\nabla u|^p - a_3\, |u|^r - \frac{1}{p^\ast}\, |u|^{p^\ast}\right] dx.
\]
Since $p < r < p^\ast$, it follows that the origin is a strict local minimizer of $E$. On the other hand, it also follows from \eqref{103} that $E(tu) \to - \infty$ as $t \to + \infty$ for any $u \in W^{1,\,p}_0(\Omega) \setminus \set{0}$. So $E$ has the mountain pass geometry. Let
\[
\Gamma = \set{\gamma \in C([0,1],W^{1,\,p}_0(\Omega)) : \gamma(0) = 0,\, E(\gamma(1)) < 0}
\]
be the class of paths in $W^{1,\,p}_0(\Omega)$ joining the origin to the set $\set{u \in W^{1,\,p}_0(\Omega) : E(u) < 0}$, and set
\begin{equation} \label{108}
c := \inf_{\gamma \in \Gamma}\, \max_{u \in \gamma([0,1])}\, E(u).
\end{equation}
Since the origin is a strict local minimizer of $E$, $c > 0$. A standard deformation argument then shows that $E$ has a \PS{c} sequence. The purpose of this paper is to give lower bounds on $F$ to guarantee that $c < c^\ast$ holds and hence this \PS{c} sequence has a subsequence that converges weakly to a nontrivial solution of problem \eqref{101}.

We assume that there is a ball $B_\rho(x_0) \subset \Omega$ such that
\begin{equation} \label{109}
F(x,t) \ge bt^s \quad \text{for a.a.\! } x \in B_\rho(x_0) \text{ and all } t \ge 0
\end{equation}
for some constants $b > 0$ and $s \in (q,p^\ast)$.

\begin{theorem} \label{Theorem 101}
Let $1 < q < p < N$ and assume \eqref{102}, \eqref{103}, \eqref{107}, and \eqref{109}. Then problem \eqref{101} has a nontrivial weak solution in each of the following cases:
\begin{enumroman}
\item $q < N(p - 1)/(N - 1)$ and $s > N^2 (p - 1)/(N - 1)(N - p)$,
\item $q \ge N(p - 1)/(N - 1)$ and $s > Nq/(N - p)$.
\end{enumroman}
\end{theorem}

\begin{remark}
We note that the two cases in Theorem \ref{Theorem 101} can be combined as
\[
s > \max \set{\frac{N^2 (p - 1)}{(N - 1)(N - p)},\frac{Nq}{N - p}}.
\]
\end{remark}

In particular, we have the following corollary for the model problem
\begin{equation} \label{110}
\left\{\begin{aligned}
- \Delta_p\, u - \Delta_q\, u & = b\, |u|^{s-2}\, u + |u|^{p^\ast - 2}\, u && \text{in } \Omega\\[10pt]
u & = 0 && \text{on } \bdry{\Omega},
\end{aligned}\right.
\end{equation}
where $1 < p < N$.

\begin{corollary} \label{Corollary 103}
Problem \eqref{110} has a nontrivial weak solution for all $b > 0$ in each of the following cases:
\begin{enumroman}
\item $1 < q < N(p - 1)/(N - 1)$ and $N^2 (p - 1)/(N - 1)(N - p) < s < p^\ast$,
\item $N(p - 1)/(N - 1) \le q < p$ and $Nq/(N - p) < s < p^\ast$.
\end{enumroman}
\end{corollary}

\begin{remark}
It was shown in Yin and Yang \cite{MR2890966} that problem \eqref{110} has a nontrivial solution when $p < s < p^\ast$ and $b > 0$ is sufficiently large. In contrast, Corollary \ref{Corollary 103} allows $s \le p$ and gives a nontrivial solution for all $b > 0$. It also gives a nontrivial solution for all $s \in (p,p^\ast)$ and $b > 0$ when $N^2 - p\, (p + 1) N + p^2 \ge 0$ and $q \le (N - p)\, p/N$, and for $s = p$ and all $b > 0$ when $N^2 - p\, (p + 1) N + p^2 > 0$ and $q < (N - p)\, p/N$.
\end{remark}

When $p \le 2 - 1/N$, case ({\em i}) in Corollary \ref{Corollary 103} cannot hold and the first inequality in case ({\em ii}) holds for $q > 1$, so we have the following corollary.

\begin{corollary} \label{Corollary 105}
If $1 < q < p \le 2 - 1/N$ and $Nq/(N - p) < s < Np/(N - p)$, then problem \eqref{110} has a nontrivial weak solution for all $b > 0$.
\end{corollary}

For the borderline case $s = q$ of problem \eqref{110} we prove a Poho\v{z}aev type nonexistence result. Recall that the corresponding nonexistence result for the $p$-Laplacian states that the problem
\[
\left\{\begin{aligned}
- \Delta_p\, u & = \lambda\, |u|^{p-2}\, u + |u|^{p^\ast-2}\, u && \text{in } \Omega\\[10pt]
u & = 0 && \text{on } \bdry{\Omega}
\end{aligned}\right.
\]
has no nontrivial weak solution in $W^{1,\,p}_0(\Omega)$ for $\lambda \le 0$ when $\Omega$ is a star-shaped domain with $C^1$-boundary (see Guedda and V{\'e}ron \cite[Corollaries 1.2 \& 1.3]{MR1009077}). In contrast, we will show that the problem
\begin{equation} \label{111}
\left\{\begin{aligned}
- \Delta_p\, u - \Delta_q\, u & = \mu\, |u|^{q-2}\, u + |u|^{p^\ast - 2}\, u && \text{in } \Omega\\[10pt]
u & = 0 && \text{on } \bdry{\Omega}
\end{aligned}\right.
\end{equation}
has no nontrivial weak solution even for small positive $\mu$.

\begin{theorem} \label{Theorem 106}
Let $1 < q < p < N$. If $\Omega$ is a star-shaped domain with $C^1$-boundary and
\begin{equation} \label{112}
\mu \le \frac{N(p - q)}{N(p - q) + pq}\, \mu_1,
\end{equation}
then problem \eqref{111} has no nontrivial weak solution in $W^{1,\,p}_0(\Omega) \cap W^{2,\,p}(\Omega)$.
\end{theorem}

\begin{remark}
It was shown in Li and Zhang \cite{MR2509998} that problem \eqref{110} has infinitely many solutions when $1 < s < q$ and $b > 0$ is sufficiently small. Theorem \ref{Theorem 106} shows that such a result cannot hold in general in the borderline case $s = q$.
\end{remark}

To prove Theorem \ref{Theorem 106} we will first derive a Poho\v{z}aev type identity for the $(p,q)$-Laplacian operator that is of independent interest. Consider the problem
\begin{equation} \label{113}
\left\{\begin{aligned}
- \Delta_p\, u - \Delta_q\, u & = g(u) && \text{in } \Omega\\[10pt]
u & = 0 && \text{on } \bdry{\Omega},
\end{aligned}\right.
\end{equation}
where $1 < q < p < N$ and $g$ is a continuous function on $\R$. Let $G(t) = \int_0^t g(\tau)\, d\tau$ be the primitive of $g$.

\begin{theorem} \label{Theorem 108}
If $\Omega$ has $C^1$-boundary and $u \in W^{1,\,p}_0(\Omega) \cap W^{2,\,p}(\Omega)$ is a weak solution of problem \eqref{113}, then
\begin{multline} \label{114}
\left(\frac{1}{q} - \frac{1}{p}\right) \int_\Omega |\nabla u|^q\, dx - \int_\Omega \left[G(u) - \frac{1}{p^\ast}\, u\, g(u)\right] dx\\[7.5pt]
+ \frac{1}{N} \int_{\bdry{\Omega}} \left[\left(1 - \frac{1}{p}\right) \abs{\frac{\partial u}{\partial \nu}}^p + \left(1 - \frac{1}{q}\right) \abs{\frac{\partial u}{\partial \nu}}^q\right](x \cdot \nu)\, d\sigma = 0,
\end{multline}
where $\nu$ is the exterior unit normal to $\bdry{\Omega}$.
\end{theorem}

Finally we prove a stronger existence result for the related problem
\begin{equation} \label{115}
\left\{\begin{aligned}
- \Delta_p\, u - \nu\, \Delta_q\, u & = f(x,u) + |u|^{p^\ast - 2}\, u && \text{in } \Omega\\[10pt]
u & = 0 && \text{on } \bdry{\Omega}
\end{aligned}\right.
\end{equation}
when the parameter $\nu > 0$ is sufficiently small.

\begin{theorem} \label{Theorem 109}
Let $1 < q < p < N$ and assume \eqref{102}, \eqref{103}, \eqref{107}, and \eqref{109}. Then there exists $\nu_0 > 0$ such that problem \eqref{115} has a nontrivial weak solution for all $\nu \in (0,\nu_0)$ in each of the following cases:
\begin{enumroman}
\item $N \ge p^2$ and $q < s < p^\ast$,
\item $N < p^2$ and either $q < s < p$ or $(Np - 2N + p)\, p/(N - p)(p - 1) < s < p^\ast$.
\end{enumroman}
\end{theorem}

\begin{remark}
We note that $p < (Np - 2N + p)\, p/(N - p)(p - 1)$ when $N < p^2$.
\end{remark}

In particular, we have the following corollary for the model problem
\begin{equation} \label{116}
\left\{\begin{aligned}
- \Delta_p\, u - \nu\, \Delta_q\, u & = b\, |u|^{s-2}\, u + |u|^{p^\ast - 2}\, u && \text{in } \Omega\\[10pt]
u & = 0 && \text{on } \bdry{\Omega},
\end{aligned}\right.
\end{equation}
where $1 < q < p < N$.

\begin{corollary} \label{Corollary 111}
There exists $\nu_0 > 0$ such that problem \eqref{116} has a nontrivial weak solution for all $\nu \in (0,\nu_0)$ and $b > 0$ in each of the following cases:
\begin{enumroman}
\item $N \ge p^2$ and $q < s < p^\ast$,
\item $N < p^2$ and either $q < s < p$ or $(Np - 2N + p)\, p/(N - p)(p - 1) < s < p^\ast$.
\end{enumroman}
\end{corollary}

When $q < s < p$, we have the following corollary.

\begin{corollary} \label{Corollary 112}
If $q < s < p$, then there exists $\nu_0 > 0$ such that problem \eqref{116} has a nontrivial weak solution for all $\nu \in (0,\nu_0)$ and $b > 0$.
\end{corollary}

\begin{remark}
A rescaling of a result in Ho and Sim \cite{MR4131823} shows that problem \eqref{116} has a nontrivial solution when $q < s < p$ and $\nu, b > 0$ are sufficiently small. In contrast, Corollary \ref{Corollary 112} gives a nontrivial solution for all $b > 0$.
\end{remark}

\section{Preliminaries}

\subsection{A compactness result}

For $\nu \ge 0$, set
\[
E_\nu(u) = \int_\Omega \left(\frac{1}{p}\, |\nabla u|^p + \frac{\nu}{q}\, |\nabla u|^q - F(x,u) - \frac{1}{p^\ast}\, |u|^{p^\ast}\right) dx, \quad u \in W^{1,\,p}_0(\Omega).
\]
Our existence results will be based on the following proposition, which extends Gazzola and Ruf \cite[Lemma 1]{MR1441856} and Arioli and Gazzola \cite[Lemma 1]{MR1741848} to the $(p,q)$-Laplacian.

\begin{proposition} \label{Proposition 201}
Let $1 < q < p < N$ and assume \eqref{103}. If $0 < c < c^\ast$, then every {\em \PS{c}} sequence has a subsequence that converges weakly to a nontrivial critical point of $E_\nu$.
\end{proposition}

\begin{proof}
Let $\seq{u_j} \subset W^{1,\,p}_0(\Omega)$ be a \PS{c} sequence, i.e.,
\begin{equation} \label{201}
E_\nu(u_j) = \int_\Omega \left(\frac{1}{p}\, |\nabla u_j|^p + \frac{\nu}{q}\, |\nabla u_j|^q - F(x,u_j) - \frac{1}{p^\ast}\, |u_j|^{p^\ast}\right) dx = c + \o(1)
\end{equation}
and
\begin{multline} \label{202}
\dualp{E_\nu'(u_j)}{v} = \int_\Omega \big(|\nabla u_j|^{p-2}\, \nabla u_j \cdot \nabla v + \nu\, |\nabla u_j|^{q-2}\, \nabla u_j \cdot \nabla v - f(x,u_j)\, v\\[7.5pt]
- |u_j|^{p^\ast - 2}\, uv\big)\, dx = \o(\norm{v}) \quad \forall v \in W^{1,\,p}_0(\Omega).
\end{multline}
Taking $v = u_j$ in \eqref{202} gives
\begin{equation} \label{203}
\int_\Omega \left(|\nabla u_j|^p + \nu\, |\nabla u_j|^q - f(x,u_j)\, u_j - |u_j|^{p^\ast}\right) dx = \o(\norm{u_j}).
\end{equation}
Fix $\sigma \in (p,p^\ast)$. Dividing \eqref{203} by $\sigma$ and subtracting from \eqref{201} gives
\begin{multline*}
\left(\frac{1}{p} - \frac{1}{\sigma}\right) \int_\Omega |\nabla u_j|^p\, dx + \left(\frac{1}{q} - \frac{1}{\sigma}\right) \nu \int_\Omega |\nabla u_j|^q\, dx\\[7.5pt]
+ \int_\Omega \left[\left(\frac{1}{\sigma} - \frac{1}{p^\ast}\right) |u_j|^{p^\ast} - F(x,u_j) + \frac{1}{\sigma}\, f(x,u_j)\, u_j\right] dx = c + \o(1) + \o(\norm{u_j}).
\end{multline*}
Since $q < p < \sigma < p^\ast$, it follows from this and \eqref{103} that $\seq{u_j}$ is bounded in $W^{1,\,p}_0(\Omega)$. So a renamed subsequence converges to some $u$ weakly in $W^{1,\,p}_0(\Omega)$, strongly in $L^r(\Omega)$, and a.e.\! in $\Omega$. Then $u$ is a critical point of $E_\nu$ by the weak continuity of $E_\nu'$ (see Li and Zhang \cite[Lemma 2.3]{MR2509998}).

Suppose $u = 0$. Then \eqref{201} and \eqref{203} reduce to
\begin{equation} \label{204}
\int_\Omega \left(\frac{1}{p}\, |\nabla u_j|^p + \frac{\nu}{q}\, |\nabla u_j|^q - \frac{1}{p^\ast}\, |u_j|^{p^\ast}\right) dx = c + \o(1)
\end{equation}
and
\begin{equation} \label{205}
\int_\Omega \left(|\nabla u_j|^p + \nu\, |\nabla u_j|^q - |u_j|^{p^\ast}\right) dx = \o(1),
\end{equation}
respectively. Equation \eqref{205} together with \eqref{105} gives
\begin{equation} \label{206}
\norm{u_j}^p \le \frac{\norm{u_j}^{p^\ast}}{S^{p^\ast/p}} + \o(1).
\end{equation}
If $\norm{u_j} \to 0$ for a renamed subsequence, then \eqref{204} gives $c = 0$, contrary to our assumption that $c > 0$. So $\norm{u_j}$ is bounded away from zero and hence \eqref{206} implies that
\[
\norm{u_j}^p \ge S^{N/p} + \o(1).
\]
Now dividing \eqref{205} by $p^\ast$ and subtracting from \eqref{204} gives
\[
c = \int_\Omega \left[\left(\frac{1}{p} - \frac{1}{p^\ast}\right) |\nabla u_j|^p + \left(\frac{1}{q} - \frac{1}{p^\ast}\right) \nu |\nabla u_j|^q\right] dx + \o(1) \ge \frac{1}{N}\, S^{N/p} + \o(1),
\]
so $c \ge c^\ast$, contrary to assumption.
\end{proof}

\subsection{Some estimates}

Let $\rho > 0$ be as in \eqref{109}, take a cut-off function $\psi \in C^\infty_0(B_\rho(0))$ such that $0 \le \psi \le 1$ and $\psi = 1$ on $B_{\rho/2}(0)$, and set
\[
u_\eps(x) = \frac{\psi(x)}{\left(\eps^{p/(p-1)} + |x|^{p/(p-1)}\right)^{(N-p)/p}}, \qquad v_\eps(x) = \frac{u_\eps(x)}{\pnorm[p^\ast]{u_\eps}}
\]
for $\eps > 0$, where $\pnorm[p^\ast]{\cdot}$ denotes the norm in $L^{p^\ast}(\Omega)$. Then $\pnorm[p^\ast]{v_\eps} = 1$. Recall that
\[
f(\eps) = \Theta(g(\eps))
\]
as $\eps \to 0$ if there exist constants $c, C > 0$ such that
\[
c\, |g(\eps)| \le |f(\eps)| \le C\, |g(\eps)|
\]
for all sufficiently small $\eps > 0$. We have the estimates
\begin{equation} \label{207}
\int_{\R^N} |\nabla v_\eps|^p\, dx = S + \Theta\big(\eps^{(N-p)/(p-1)}\big),
\end{equation}
where $S$ is as in \eqref{105},
\begin{equation} \label{208}
\int_{\R^N} |\nabla v_\eps|^q\, dx = \begin{cases}
\Theta\big(\eps^{N(p-q)/p}\big), & q > \frac{N(p - 1)}{N - 1}\\[7.5pt]
\Theta\big(\eps^{N(N-p)/(N-1)p}\, |\log \eps|\big), & q = \frac{N(p - 1)}{N - 1}\\[7.5pt]
\Theta\big(\eps^{(N-p)q/p(p-1)}\big), & q < \frac{N(p - 1)}{N - 1},
\end{cases}
\end{equation}
and
\begin{equation} \label{209}
\int_{\R^N} v_\eps^s\, dx = \begin{cases}
\Theta\big(\eps^{[Np-(N-p)s]/p}\big), & s > \frac{N(p - 1)}{N - p}\\[7.5pt]
\Theta\big(\eps^{N/p}\, |\log \eps|\big), & s = \frac{N(p - 1)}{N - p}\\[7.5pt]
\Theta\big(\eps^{(N-p)s/p(p-1)}\big), & s < \frac{N(p - 1)}{N - p}
\end{cases}
\end{equation}
as $\eps \to 0$ (see Dr{\'a}bek and Huang \cite{MR1473856}).

For $\eps > 0$ and $0 < \delta \le 1$, set
\[
u_{\eps,\delta}(x) = \frac{\psi(x/\delta)}{\left(\eps^{p/(p-1)} + |x|^{p/(p-1)}\right)^{(N-p)/p}}, \qquad v_{\eps,\delta}(x) = \frac{u_{\eps,\delta}(x)}{\pnorm[p^\ast]{u_{\eps,\delta}}}.
\]
Then $\pnorm[p^\ast]{v_{\eps,\delta}} = 1$ and we will derive estimates similar to \eqref{207}--\eqref{209} for $v_{\eps,\delta}$. First we note that
\begin{equation} \label{210}
u_{\eps,\delta}(x) = \delta^{-(N-p)/(p-1)}\, u_{\eps/\delta}(x/\delta).
\end{equation}
So
\[
\int_{\R^N} u_{\eps,\delta}^{p^\ast}\, dx = \delta^{-Np/(p-1)} \int_{\R^N} u_{\eps/\delta}(x/\delta)^{p^\ast} dx = \delta^{-N/(p-1)} \int_{\R^N} u_{\eps/\delta}^{p^\ast}\, dx
\]
and hence
\begin{equation} \label{211}
\pnorm[p^\ast]{u_{\eps,\delta}} = \delta^{-(N-p)/p(p-1)} \pnorm[p^\ast]{u_{\eps/\delta}}.
\end{equation}
It follows from \eqref{210} and \eqref{211} that
\[
v_{\eps,\delta}(x) = \delta^{-(N-p)/p}\, v_{\eps/\delta}(x/\delta).
\]
So
\begin{equation} \label{212}
\int_{\R^N} v_{\eps,\delta}^s\, dx = \delta^{-(N-p)s/p} \int_{\R^N} v_{\eps/\delta}(x/\delta)^s\, dx = \delta^{[Np-(N-p)s]/p} \int_{\R^N} v_{\eps/\delta}^s\, dx.
\end{equation}
Moreover,
\[
\nabla v_{\eps,\delta}(x) = \delta^{-N/p}\, \nabla v_{\eps/\delta}(x/\delta)
\]
and hence
\begin{equation} \label{213}
\int_{\R^N} |\nabla v_{\eps,\delta}|^q\, dx = \delta^{-Nq/p} \int_{\R^N} |\nabla v_{\eps/\delta}(x/\delta)|^q\, dx = \delta^{N(p-q)/p} \int_{\R^N} |\nabla v_{\eps/\delta}|^q\, dx.
\end{equation}
Combining \eqref{212} and \eqref{213} with \eqref{207}--\eqref{209} gives us the following estimates.

\begin{lemma} \label{Lemma 202}
As $\eps \to 0$ and $\eps/\delta \to 0$,
\begin{equation} \label{214}
\int_{\R^N} |\nabla v_{\eps,\delta}|^p\, dx = S + \Theta\big((\eps/\delta)^{(N-p)/(p-1)}\big),
\end{equation}
\begin{equation} \label{215}
\int_{\R^N} |\nabla v_{\eps,\delta}|^q\, dx = \begin{cases}
\Theta\big(\eps^{N(p-q)/p}\big), & q > \frac{N(p - 1)}{N - 1}\\[7.5pt]
\Theta\big(\eps^{N(N-p)/(N-1)p}\, |\log\, (\eps/\delta)|\big), & q = \frac{N(p - 1)}{N - 1}\\[7.5pt]
\Theta\big(\eps^{(N-p)q/p(p-1)}\, \delta^{[N(p-1)-(N-1)q]/(p-1)}\big), & q < \frac{N(p - 1)}{N - 1},
\end{cases}
\end{equation}
and
\begin{equation} \label{216}
\int_{\R^N} v_{\eps,\delta}^s\, dx = \begin{cases}
\Theta\big(\eps^{[Np-(N-p)s]/p}\big), & s > \frac{N(p - 1)}{N - p}\\[7.5pt]
\Theta\big(\eps^{N/p}\, |\log\, (\eps/\delta)|\big), & s = \frac{N(p - 1)}{N - p}\\[7.5pt]
\Theta\big(\eps^{(N-p)s/p(p-1)}\, \delta^{[N(p-1)-(N-p)s]/(p-1)}\big), & s < \frac{N(p - 1)}{N - p}.
\end{cases}
\end{equation}
\end{lemma}

Next we prove the following proposition.

\begin{proposition} \label{Proposition 203}
If $\seq{\eps_j}, \seq{\delta_j}$ are sequences such that $\eps_j \to 0,\, 0 < \delta_j \le 1,\, \eps_j/\delta_j \to 0$,
\begin{equation} \label{217}
\frac{\nu \dint_{\R^N} |\nabla v_{\eps_j,\delta_j}|^q\, dx}{\dint_{\R^N} v_{\eps_j,\delta_j}^s\, dx} \to 0, \qquad \frac{(\eps_j/\delta_j)^{(N-p)/(p-1)}}{\dint_{\R^N} v_{\eps_j,\delta_j}^s\, dx} \to 0,
\end{equation}
then
\[
\max_{t \ge 0}\, E_\nu(t v_{\eps_j,\delta_j}(x - x_0)) < c^\ast
\]
for all sufficiently large $j$.
\end{proposition}

\begin{proof}
Write $v_j(x) = v_{\eps_j,\delta_j}(x - x_0)$. Since $v_j(x) = 0$ for all $x \in \Omega \setminus B_\rho(x_0)$,
\[
F(x,t v_j(x)) \ge bt^s v_j(x)^s \quad \text{for a.a.\! } x \in \Omega \text{ and all } t \ge 0
\]
by \eqref{109}, so
\[
E_\nu(t v_j) \le \frac{t^p}{p} \int_\Omega |\nabla v_j|^p\, dx + \frac{\nu t^q}{q} \int_\Omega |\nabla v_j|^q\, dx - bt^s \int_\Omega v_j^s\, dx - \frac{t^{p^\ast}}{p^\ast} =: \varphi(t).
\]
Suppose that the conclusion of the lemma is false. Then there are renamed subsequences $\seq{\eps_j}, \seq{\delta_j}$ and $t_j > 0$ such that
\begin{equation} \label{218}
\varphi(t_j) = \frac{t_j^p}{p} \int_\Omega |\nabla v_j|^p\, dx + \frac{\nu t_j^q}{q} \int_\Omega |\nabla v_j|^q\, dx - b t_j^s \int_\Omega v_j^s\, dx - \frac{t_j^{p^\ast}}{p^\ast} \ge c^\ast
\end{equation}
and
\begin{equation} \label{219}
t_j\, \varphi'(t_j) = t_j^p \int_\Omega |\nabla v_j|^p\, dx + \nu t_j^q \int_\Omega |\nabla v_j|^q\, dx - sb t_j^s \int_\Omega v_j^s\, dx - t_j^{p^\ast} = 0.
\end{equation}
By Lemma \ref{Lemma 202},
\[
\int_\Omega |\nabla v_j|^p\, dx \to S, \qquad \int_\Omega |\nabla v_j|^q\, dx \to 0, \qquad \int_\Omega v_j^s\, dx \to 0.
\]
So \eqref{218} implies that the sequence $\seq{t_j}$ is bounded and hence converges to some $t_0 > 0$ for a subsequence. Passing to the limit in \eqref{219} gives
\begin{equation} \label{220}
S t_0^p - t_0^{p^\ast} = 0,
\end{equation}
so $t_0 = S^{(N-p)/p^2}$.

Subtracting \eqref{220} from \eqref{219} and using \eqref{214} gives
\[
S \big(t_j^p - t_0^p\big) + \nu t_j^q \int_\Omega |\nabla v_j|^q\, dx - sb t_j^s \int_\Omega v_j^s\, dx - \big(t_j^{p^\ast} - t_0^{p^\ast}\big) = \Theta\big((\eps_j/\delta_j)^{(N-p)/(p-1)}\big).
\]
Then
\begin{equation} \label{221}
\left(p\, S \sigma_j^{p-1} - p^\ast \tau_j^{p^\ast - 1}\right) (t_j - t_0) = sb t_j^s \int_\Omega v_j^s\, dx - \nu t_j^q \int_\Omega |\nabla v_j|^q\, dx + \Theta\big((\eps_j/\delta_j)^{(N-p)/(p-1)}\big)
\end{equation}
for some $\sigma_j$ and $\tau_j$ between $t_0$ and $t_j$ by the mean value theorem. Since $t_j \to t_0$, $\sigma_j, \tau_j \to t_0$ and hence
\[
p\, S \sigma_j^{p-1} - p^\ast \tau_j^{p^\ast - 1} \to p\, S t_0^{p-1} - p^\ast t_0^{p^\ast - 1} = - (p^\ast - p)\, t_0^{p^\ast - 1}
\]
by \eqref{220}. So \eqref{221} together with \eqref{217} gives
\[
t_j = t_0 - \left(\frac{sb t_0^{- (p^\ast - s - 1)}}{p^\ast - p} + \o(1)\right) \int_\Omega v_j^s\, dx < t_0
\]
for all sufficiently large $j$.

Dividing \eqref{219} by $p^\ast$, subtracting from \eqref{218}, using \eqref{214}, and writing $c^\ast$ in terms of $t_0$ gives
\[
\frac{1}{N}\, S t_j^p + \left(\frac{1}{q} - \frac{1}{p^\ast}\right) \nu t_j^q \int_\Omega |\nabla v_j|^q\, dx - b \left(1 - \frac{s}{p^\ast}\right) t_j^s \int_\Omega v_j^s\, dx \ge \frac{1}{N}\, S t_0^p + \Theta\big((\eps_j/\delta_j)^{(N-p)/(p-1)}\big).
\]
This together with $t_j < t_0$ and \eqref{217} gives
\[
b \left(1 - \frac{s}{p^\ast}\right) t_0^s \le 0,
\]
a contradiction since $s < p^\ast$ and $t_0 > 0$.
\end{proof}

\section{Proofs}

\subsection{Proof of Theorem \ref{Theorem 101}}

Lemma \ref{Lemma 202} gives the following estimates for the quotients in \eqref{217}.

\begin{lemma} \label{Lemma 301}
If $s > N(p - 1)/(N - p)$, then
\[
\frac{\dint_{\R^N} |\nabla v_{\eps_j,\delta_j}|^q\, dx}{\dint_{\R^N} v_{\eps_j,\delta_j}^s\, dx} = \begin{cases}
\Theta\big(\eps_j^{[(N-p)s-Nq]/p}\big), & q > \frac{N(p - 1)}{N - 1}\\[7.5pt]
\Theta\big(\eps_j^{[(N-p)s-Nq]/p}\, |\log\, (\eps_j/\delta_j)|\big), & q = \frac{N(p - 1)}{N - 1}\\[7.5pt]
\Theta\big(\eps_j^{[(N-p)(s+q/(p-1))-Np]/p}\, \delta_j^{[N(p-1)-(N-1)q]/(p-1)}\big), & q < \frac{N(p - 1)}{N - 1}
\end{cases}
\]
and
\[
\frac{(\eps_j/\delta_j)^{(N-p)/(p-1)}}{\dint_{\R^N} v_{\eps_j,\delta_j}^s\, dx} = \Theta\big(\eps_j^{[(N-p)(p-1)s-(Np-2N+p)p]/p(p-1)}\, \delta_j^{-(N-p)/(p-1)}\big).
\]
\end{lemma}

We are now ready to prove Theorem \ref{Theorem 101}.

\begin{proof}[Proof of Theorem \ref{Theorem 101}]
As we have noted in the introduction, it suffices to show that the mountain pass level $c$ defined in \eqref{108} is below the threshold level $c^\ast$ in \eqref{104}. For any $u \in W^{1,\,p}_0(\Omega) \setminus \set{0}$, $E(tu) \to - \infty$ as $t \to + \infty$ and hence $\exists\, t_u > 0$ such that $E(t_u u) < 0$. Then the line segment $\set{tu : 0 \le t \le t_u}$ belongs to $\Gamma$ and hence
\begin{equation} \label{301}
c \le \max_{0 \le t \le t_u}\, E(tu) \le \max_{t \ge 0}\, E(tu).
\end{equation}
In each of the two cases in the theorem, we will construct sequences $\seq{\eps_j}, \seq{\delta_j}$ such that $\eps_j \to 0,\, 0 < \delta_j \le 1,\, \eps_j/\delta_j \to 0$, and \eqref{217} with $\nu = 1$ holds, and conclude from Proposition \ref{Proposition 203} and \eqref{301} that $c < c^\ast$.

({\em i}) Let $q < N(p - 1)/(N - 1)$ and $s > N^2 (p - 1)/(N - 1)(N - p)$. We take a sequence $\eps_j \to 0$ and set $\delta_j = \eps_j^\kappa$, where $\kappa \in [0,1)$ is to be determined. Since
\[
s > \frac{N^2 (p - 1)}{(N - 1)(N - p)} > \frac{N(p - 1)}{N - p},
\]
Lemma \ref{Lemma 301} gives
\begin{align*}
\frac{\dint_{\R^N} |\nabla v_{\eps_j,\delta_j}|^q\, dx}{\dint_{\R^N} v_{\eps_j,\delta_j}^s\, dx} & = \Theta\big(\eps_j^{[(N-p)(s+q/(p-1))-Np]/p+\kappa[N(p-1)-(N-1)q]/(p-1)}\big)\\[5pt]
& = \Theta\big(\eps_j^{[N(p-1)-(N-1)q](\kappa - \underline{\kappa})/(p-1)}\big),
\end{align*}
where
\[
\underline{\kappa} = \frac{Np(p-1)-(N-p)(p-1)s-(N-p)q}{[N(p-1)-(N-1)q]p},
\]
and
\begin{align*}
\frac{(\eps_j/\delta_j)^{(N-p)/(p-1)}}{\dint_{\R^N} v_{\eps_j,\delta_j}^s\, dx} & = \Theta\big(\eps_j^{[(N-p)(p-1)s-(Np-2N+p)p]/p(p-1)-\kappa(N-p)/(p-1)}\big)\\[5pt]
& = \Theta\big(\eps_j^{(N-p)(\overline{\kappa} - \kappa)/(p-1)}\big),
\end{align*}
where
\[
\overline{\kappa} = \frac{(N-p)(p-1)s-(Np-2N+p)p}{(N-p)p}.
\]
We want to choose $\kappa \in [0,1)$ so that $\kappa > \underline{\kappa}$ and $\kappa < \overline{\kappa}$. This is possible if and only if $\underline{\kappa} < \overline{\kappa}$, $\underline{\kappa} < 1$, and $\overline{\kappa} > 0$. Tedious calculations show that these inequalities are equivalent to
\[
s > \frac{N^2 (p - 1)}{(N - 1)(N - p)},
\]
\[
s > \frac{Nq}{N - p},
\]
and
\[
s > \frac{N^2 (p - 1)}{(N - 1)(N - p)} - \frac{N - p}{(N - 1)(p - 1)},
\]
respectively, all of which hold under our assumptions on $q$ and $s$.

({\em ii}) Let $q \ge N(p - 1)/(N - 1)$ and $s > Nq/(N - p)$. We take a sequence $\eps_j \to 0$ and set $\delta_j = 1$. Since
\[
s > \frac{Nq}{N - p} \ge \frac{N^2 (p - 1)}{(N - 1)(N - p)} > \frac{N(p - 1)}{N - p},
\]
Lemma \ref{Lemma 301} gives
\[
\frac{\dint_{\R^N} |\nabla v_{\eps_j,\delta_j}|^q\, dx}{\dint_{\R^N} v_{\eps_j,\delta_j}^s\, dx} = \begin{cases}
\Theta\big(\eps_j^{[(N-p)s-Nq]/p}\big), & q > \frac{N(p - 1)}{N - 1}\\[7.5pt]
\Theta\big(\eps_j^{[(N-p)s-Nq]/p}\, |\log \eps_j|\big), & q = \frac{N(p - 1)}{N - 1}
\end{cases}
\]
and
\[
\frac{(\eps_j/\delta_j)^{(N-p)/(p-1)}}{\dint_{\R^N} v_{\eps_j,\delta_j}^s\, dx} = \Theta\big(\eps_j^{[(N-p)(p-1)s-(Np-2N+p)p]/p(p-1)}\big).
\]
Since $s > Nq/(N - p)$, the first limit in \eqref{217} holds. The second limit also holds since
\[
\frac{Nq}{N - p} \ge \frac{N^2 (p - 1)}{(N - 1)(N - p)} > \frac{(Np - 2N + p) p}{(N - p)(p - 1)}. \QED
\]
\end{proof}

\subsection{Proofs of Theorems \ref{Theorem 106} and \ref{Theorem 108}}

First we prove Theorem \ref{Theorem 108}.

\begin{proof}[Proof of Theorem \ref{Theorem 108}]
Integrating the easily verified identity
\begin{multline*}
\left[\divg \left(|\nabla u|^{p-2}\, \nabla u + |\nabla u|^{q-2}\, \nabla u\right) + g(u)\right](x \cdot \nabla u) = \left(\frac{N}{p} - 1\right) |\nabla u|^p + \left(\frac{N}{q} - 1\right) |\nabla u|^q\\[7.5pt]
- NG(u) + \divg \left[\left(|\nabla u|^{p-2}\, \nabla u + |\nabla u|^{q-2}\, \nabla u\right)(x \cdot \nabla u) - x \left(\frac{|\nabla u|^p}{p} + \frac{|\nabla u|^q}{q}\right) + x\, G(u)\right]
\end{multline*}
over $\Omega$ gives
\begin{multline*}
\left(\frac{N}{p} - 1\right) \int_\Omega |\nabla u|^p\, dx + \left(\frac{N}{q} - 1\right) \int_\Omega |\nabla u|^q\, dx - N \int_\Omega G(u)\, dx\\[7.5pt]
+ \int_{\bdry{\Omega}} \left[\left(|\nabla u|^{p-2}\, \nabla u + |\nabla u|^{q-2}\, \nabla u\right)(x \cdot \nabla u) - x \left(\frac{|\nabla u|^p}{p} + \frac{|\nabla u|^q}{q}\right)\right] \cdot \nu\, d\sigma = 0
\end{multline*}
since $u$ is a weak solution of problem \eqref{113}. We have $(\nabla u \cdot \nu)(x \cdot \nabla u) = |\nabla u|^2\, (x \cdot \nu)$ and $|\nabla u| = \abs{\dfrac{\partial u}{\partial \nu}}$ since $u = 0$ on $\bdry{\Omega}$, so the last equation reduces to
\begin{multline} \label{302}
\left(\frac{N}{p} - 1\right) \int_\Omega |\nabla u|^p\, dx + \left(\frac{N}{q} - 1\right) \int_\Omega |\nabla u|^q\, dx - N \int_\Omega G(u)\, dx\\[7.5pt]
+ \int_{\bdry{\Omega}} \left[\left(1 - \frac{1}{p}\right) \abs{\frac{\partial u}{\partial \nu}}^p + \left(1 - \frac{1}{q}\right) \abs{\frac{\partial u}{\partial \nu}}^q\right](x \cdot \nu)\, d\sigma = 0.
\end{multline}
On the other hand, testing problem \eqref{113} with $u$ gives
\begin{equation} \label{303}
\int_\Omega |\nabla u|^p\, dx + \int_\Omega |\nabla u|^q\, dx - \int_\Omega u\, g(u)\, dx = 0.
\end{equation}
Multiplying \eqref{303} by $N/p - 1$ and subtracting from \eqref{302} gives \eqref{114}.
\end{proof}

Now we prove Theorem \ref{Theorem 106}.

\begin{proof}[Proof of Theorem \ref{Theorem 106}]
Suppose problem \eqref{111} has a nontrivial weak solution $u \in W^{1,\,p}_0(\Omega) \cap W^{2,\,p}(\Omega)$. Taking $g(t) = \mu\, |t|^{q-2}\, t + |t|^{p^\ast - 2}\, t$ in \eqref{114} and combining with \eqref{112} and \eqref{106} gives
\begin{multline} \label{304}
\frac{1}{N} \int_{\bdry{\Omega}} \left[\left(1 - \frac{1}{p}\right) \abs{\frac{\partial u}{\partial \nu}}^p + \left(1 - \frac{1}{q}\right) \abs{\frac{\partial u}{\partial \nu}}^q\right](x \cdot \nu)\, d\sigma = \left(\frac{1}{q} - \frac{1}{p^\ast}\right) \mu \int_\Omega |u|^q\, dx\\[7.5pt]
- \left(\frac{1}{q} - \frac{1}{p}\right) \int_\Omega |\nabla u|^q\, dx \le \left(\frac{1}{q} - \frac{1}{p}\right) \left(\mu_1 \int_\Omega |u|^q\, dx - \int_\Omega |\nabla u|^q\, dx\right) \le 0.
\end{multline}
Without loss of generality, we may assume that $\Omega$ is star-shaped with respect to the origin. Then $x \cdot \nu > 0$ on $\bdry{\Omega}$, so \eqref{304} implies that $u$ is an eigenfunction of the $q$-Laplacian associated with the eigenvalue $\mu_1$ and $\partial u/\partial \nu = 0$ on $\bdry{\Omega}$, contradicting the Hopf lemma (see V{\'a}zquez \cite[Theorem 5]{MR768629}).
\end{proof}

\subsection{Proof of Theorem \ref{Theorem 109}}

We have
\[
E_\nu(u) = E_0(u) + \frac{\nu}{q} \int_\Omega |\nabla u|^q\, dx, \quad u \in W^{1,\,p}_0(\Omega).
\]
Taking $\nu = 0$ and $\delta_j = 1$ in Proposition \ref{Proposition 203} and noting that $v_{\eps,1} = v_\eps$ gives the following proposition.

\begin{proposition} \label{Proposition 302}
If
\begin{equation} \label{305}
\frac{\eps^{(N-p)/(p-1)}}{\dint_{\R^N} v_\eps^s\, dx} \to 0 \text{ as } \eps \to 0,
\end{equation}
then
\[
\max_{t \ge 0}\, E_0(t v_\eps(x - x_0)) < c^\ast
\]
for all sufficiently small $\eps > 0$.
\end{proposition}

Equation \eqref{209} gives the following estimate for the quotient in \eqref{305}.

\begin{lemma} \label{Lemma 303}
We have
\[
\frac{\eps^{(N-p)/(p-1)}}{\dint_{\R^N} v_\eps^s\, dx} = \begin{cases}
\Theta\big(\eps^{[(N-p)(p-1)s-(Np-2N+p)p]/p(p-1)}\big), & s > \frac{N(p - 1)}{N - p}\\[7.5pt]
\Theta\big(\eps^{(N-p^2)/p(p-1)}/|\log \eps|\big), & s = \frac{N(p - 1)}{N - p}\\[7.5pt]
\Theta\big(\eps^{(N-p)(p-s)/p(p-1)}\big), & s < \frac{N(p - 1)}{N - p}.
\end{cases}
\]
\end{lemma}

We are now ready to prove Theorem \ref{Theorem 109}.

\begin{proof}[Proof of Theorem \ref{Theorem 109}]
The proof is similar to that of Theorem \ref{Theorem 101}, so we will be sketchy. Let
\[
\Gamma_\nu = \set{\gamma \in C([0,1],W^{1,\,p}_0(\Omega)) : \gamma(0) = 0,\, E_\nu(\gamma(1)) < 0},
\]
set
\[
c_\nu := \inf_{\gamma \in \Gamma_\nu}\, \max_{u \in \gamma([0,1])}\, E_\nu(u),
\]
and note that $c_\nu > 0$ when $\nu > 0$. It suffices to show that $c_\nu < c^\ast$ for sufficiently small $\nu$. We will show that
\begin{equation} \label{306}
c_0 := \inf_{\gamma \in \Gamma_0}\, \max_{u \in \gamma([0,1])}\, E_0(u) < c^\ast.
\end{equation}
Then there is a path $\gamma_0 \in \Gamma_0$ such that
\[
\max_{u \in \gamma_0([0,1])}\, E_0(u) < c^\ast.
\]
For all sufficiently small $\nu > 0$,
\[
E_\nu(\gamma_0(1)) = E_0(\gamma_0(1)) + \frac{\nu}{q} \int_\Omega |\nabla \gamma_0(1)|^q\, dx < 0
\]
and
\[
\max_{u \in \gamma_0([0,1])}\, E_\nu(u) \le \max_{u \in \gamma_0([0,1])}\, E_0(u) + \frac{\nu}{q} \left(\max_{u \in \gamma_0([0,1])} \int_\Omega |\nabla u|^q\, dx\right) < c^\ast,
\]
so $\gamma_0 \in \Gamma_\nu$ and
\[
c_\nu \le \max_{u \in \gamma_0([0,1])}\, E_\nu(u) < c^\ast.
\]

To show that \eqref{306} holds, it suffices to show that
\begin{equation} \label{307}
\max_{t \ge 0}\, E_0(t u_0) < c^\ast
\end{equation}
for some $u_0 \in W^{1,\,p}_0(\Omega) \setminus \set{0}$ as in the proof of Theorem \ref{Theorem 101}. In each of the two cases in the theorem, we will show that \eqref{305} holds and conclude from Proposition \ref{Proposition 302} that \eqref{307} holds for $u_0 = v_\eps(x - x_0)$ with $\eps > 0$ sufficiently small.

({\em i}) Let $N \ge p^2$ and $q < s < p^\ast$. If $s > N(p - 1)/(N - p)$, then
\[
(N - p)(p - 1)\, s - (Np - 2N + p)\, p > N(p - 1)^2 - (Np - 2N + p)\, p = N - p^2,
\]
and if $s < N(p - 1)/(N - p)$, then
\[
(N - p)(p - s) > (N - p)\, p - N(p - 1) = N - p^2.
\]
So \eqref{305} follows from Lemma \ref{Lemma 303}.

({\em ii}) Let $N < p^2$. Then
\[
p < \frac{N(p - 1)}{N - p} < \frac{(Np - 2N + p)\, p}{(N - p)(p - 1)}.
\]
So if $q < s < p$, then $s < N(p - 1)/(N - p)$, and if $(Np - 2N + p)\, p/(N - p)(p - 1) < s < p^\ast$, then $s > N(p - 1)/(N - p)$. In either case, \eqref{305} follows from Lemma \ref{Lemma 303}.
\end{proof}

\subsection*{Acknowledgements}
The third author was supported by the  2022-0461 Research
Fund of the University of Ulsan.

\def\cdprime{$''$}

\end{document}